\documentclass[a4paper,11pt]{article}

\def\ms{\textbf{2020 Mathematics Subject Classifications}: }
\providecommand\key[1]{\par \textbf{Keywords: }#1 \vspace{0,5cm}}
\providecommand\add[2]{\textsc{{#1}} {\newline \indent \textbf{ Email(s): }#2}}

\usepackage{geometry}
\usepackage{amsmath,amsthm,amscd,amsfonts,amssymb,enumerate}
\usepackage{graphicx}		
\usepackage{epsfig}
\usepackage{epstopdf}
\usepackage{float}

\newtheorem{theorem}{Theorem}[section]
\newtheorem{lemma}[theorem]{Lemma}

\newtheorem{corollary}[theorem]{Corollary}
\newtheorem{definition}[theorem]{Definition}
\newtheorem{example}[theorem]{Example}

\begin{document}

\title{Biframes in Hilbert $C^{\ast}-$modules}

\author{Mohamed Rossafi$^{1*}$, Abdelilah Karara$^{2}$ and Roumaissae El jazzar$^{2}$ }

\date{}
\maketitle

\begin{abstract}
In this paper, we will introduce the concept of biframes for Hilbert $ C^{\ast}- $modules produced by a pair of sequences, and we present various examples of biframes. Then, we examine the characteristics of biframes from the viewpoint of operator theory by establishing some properties of biframes in Hilbert $ C^{\ast}- $modules.\\
\end{abstract}

\ms {42C15; 46C05; 47B90.}\\

\key {Frame, biframe, Hilbert $ C^{\ast}- $modules.}\\


\add {Laboratory Partial Differential Equations, Spectral Algebra and Geometry, Higher School of Education and Training, University of Ibn Tofail, Kenitra, Morocco $^{1*}$ (corresponding author) 
	\\ 
	\indent Laboratory Partial Differential Equations, Spectral Algebra and Geometry, Department of Mathematics, Faculty of Sciences, University of Ibn Tofail, Kenitra, Morocco. $^{2}$ 
	\\
}
	{mohamed.rossafi@usmba.ac.ma; mohamed.rossafi@uit.ac.ma $^{1*}$ (corresponding author),  
	abdelilah.karara@uit.ac.ma; roumaissae.eljazzar@uit.ac.ma $^{2}$ 
	}

\section{Introduction }

The notion  of frames in Hilbert spaces has been introduced by Duffin and Schaffer \cite{Duf} in 1952 to research certain difficult nonharmonic Fourier series problems, following the fundamental paper \cite{DGM} by Daubechies, Grossman and Meyer, frame theory started to become popular, especially in the more specific context of Gabor frames and wavelet frames. In recent decades, many researchers have been attracted to the frame theory and related fields. For example, 
Wenchang Sun \cite{S} introduced the generalized frame, or g-frame, for a Hilbert space. Recently, D. Li and J. Leng \cite{Li} introduced Operator representations of g-frames in Hilbert spaces. Controlled g-frames in Hilbert $ C^{\ast}- $modules have been investigated by N. K. Sahu \cite{Sah}. The notion of Approximate Oblique Dual g-Frames for Closed Subspaces of Hilbert Spaces can be found in \cite{Chi}. For more about frames, see \cite{KAI, FR1, Rah, RFDCA, r1, r3, r5, r6}.

The idea of pair frames, which refers to a pair of sequences in a Hilbert space, was first presented in \cite{Fer}. 

In 2018,  M. M.  Azandaryani and A. Fereydooni \cite{Aza} introduced pair frames in Hilbert $ C^{\ast}- $modules and we generalized some of the results obtained for pair frames in Hilbert spaces to Hilbert $ C^{\ast}- $modules. 

This paper introduces the concept of biframes, which involves examining a pair of sequences within Hilbert $C^{\ast}$-modules from a novel perspective. Unlike the traditional frame concept, which is based on a single sequence, the definition of a biframe necessitates the use of two distinct sequences. Essentially, the biframe concept can be viewed as an extension of controlled frames and a specific instance within the broader category of pair frames.

\section{Preliminaries}
Hilbert $C^{\ast}$-modules are generalizations of Hilbert spaces by allowing the inner product to take values in a $C^{\ast}$-algebra rather than in the field of complex
numbers.

Let's now review the definition of a Hilbert $C^{\ast}$-module, the basic properties and some facts concerning operators on Hilbert $C^{\ast}$-module.

\begin{definition}\cite{Kal}
	Consider $\mathcal{A}$ as a unital $C^{\ast}$-algebra and $\mathcal{H}$ as a left $\mathcal{A}$-module, where their linear structures are in alignment. Define $\mathcal{H}$ as a pre-Hilbert $\mathcal{A}$-module if there exists an $\mathcal{A}$-valued inner product $\langle \cdot, \cdot\rangle: \mathcal{H} \times \mathcal{H} \rightarrow \mathcal{A}$, characterized by being sesquilinear, positively definite, and in conformity with the module action. Specifically,
		\begin{itemize}
		\item [1]- For any $\xi \in \mathcal{H}$, $\langle \xi, \xi\rangle \geq 0$, with $\langle \xi, \xi\rangle = 0$ if and only if $\xi = 0$.
		\item [2]- For all $a \in \mathcal{A}$ and $\xi, \eta, z \in \mathcal{H}$, it holds that $\langle a\xi + \eta, z\rangle = a\langle \xi, z\rangle + \langle \eta, z\rangle$.
		\item [3]- For every $\xi, \eta \in \mathcal{H}$, the relation $\langle \xi, \eta \rangle = \langle \eta, \xi \rangle^{\ast}$ is maintained.
	\end{itemize}
	For any element $\xi$ within the set $\mathcal{H}$, its norm is defined as the square root of the norm of its $\mathcal{A}$-valued inner product with itself, mathematically expressed as $\Vert \xi\Vert = \Vert \langle \xi, \xi \rangle \Vert^{\frac{1}{2}}$. When the set $\mathcal{H}$ is complete under this norm, it is classified as a Hilbert module over $\mathcal{A}$, also known as a Hilbert $C^{\ast}$-module with respect to $\mathcal{A}$. 
	\end{definition}

Additionally, for any element $a$ within the $C^{\ast}$-algebra $\mathcal{A}$, the magnitude of $a$ can be described by the formula $\vert a \vert = (a^{\ast}a)^{\frac{1}{2}}$. Furthermore, the norm valued in $\mathcal{A}$ for each element $\xi$ in $\mathcal{H}$ is represented by $\vert \xi\vert = \langle \xi, \xi \rangle^{\frac{1}{2}}$.
	
Let $\mathcal{H}$ and $\mathcal{K}$ be two Hilbert $\mathcal{A}$-modules. A map $\mathcal{T}: \mathcal{H}\rightarrow\mathcal{K}$ is said to be adjointable if there exists a map $\mathcal{T}^{\ast}:\mathcal{K}\rightarrow\mathcal{H}$ such that $\langle \mathcal{T}\xi,\eta\rangle_{\mathcal{A}}=\langle \xi,\mathcal{T}^{\ast}\eta\rangle_{\mathcal{A}}$ for all $\xi\in\mathcal{H}$ and $\eta\in\mathcal{K}$.

We also reserve the notation $End_{\mathcal{A}}^{\ast}(\mathcal{H},\mathcal{K})$ for the set of all adjointable operators from $\mathcal{H}$ to $\mathcal{K}$ and $End_{\mathcal{A}}^{\ast}(\mathcal{H},\mathcal{H})$ is abbreviated to $End_{\mathcal{A}}^{\ast}(\mathcal{H})$. We denote $GL(\mathcal{H})$ the set of all bounded linear operators which have bounded inverse and $GL^{+}(\mathcal{H})$ the set of all bounded positive linear operators which have bounded inverse.    
Furthermore, for an operator $\mathcal{T}$ in $\operatorname{End}_\mathcal{A}^*(\mathcal{H})$, the term $\mathcal{R}(\mathcal{T})$ is used to represent its range.

The following lemma will be used to prove our mains results.
\begin{lemma}\cite{Pas}
	Assume $\mathcal{H}$ is a Hilbert $\mathcal{A}$-module. Given an operator $\mathcal{T} \in End_{\mathcal{A}}^{\ast}(\mathcal{H})$, it follows that 
	$$\langle \mathcal{T}\xi, \mathcal{T}\xi \rangle_{\mathcal{A}} \leq \|\mathcal{T}\|^2 \langle \xi, \xi \rangle_{\mathcal{A}}, \text{ for every } \xi \in \mathcal{H}.$$
\end{lemma}

\begin{definition}\cite{Lar1}
	Consider a set $\Xi=\{\xi_i\}_{i=1}^{\infty}$, which constitutes a sequence of elements within the space $\mathcal{H}$. Suppose there are two positive constants, denoted as $A$ and $B$. These constants are such that for any element $\xi \in \mathcal{H}$, the following condition holds:
	\begin{equation}
		A\langle \xi,\xi\rangle_{\mathcal{A}} \leq \sum_{i=1}^{\infty} \langle \xi,\xi_{i}\rangle_{\mathcal{A}} \langle \xi_{i},\xi\rangle_{\mathcal{A}} \leq B \langle \xi,\xi\rangle_{\mathcal{A}},
	\end{equation}
	In this context, $A$ and $B$ are referred to as the lower and upper frame bounds, respectively. A frame is considered $\lambda$-tight if $A = B = \lambda$. When $A = B = 1$, the frame is recognized as a Parseval frame.
\end{definition}
The frame operator $S: \mathcal{H} \to \mathcal{H}$ is defined by $S\xi=\langle \xi,\xi_{i}\rangle_{\mathcal{A}} \xi_{i}$, $\forall\xi\in\mathcal{H}$.

Let $\{\xi_i\}_{i=1}^{\infty}$ be a frame for $\mathcal{H}$ and $\{\eta_i\}_{i=1}^{\infty}$ be a sequence of $\mathcal{H}$. Then $\{\eta_i\}_{i=1}^{\infty}$ is called a dual sequence of $\{\xi_i\}_{i=1}^{\infty}$ if $\xi=\langle \xi,\eta_{i}\rangle_{\mathcal{A}} \xi_{i}$,  $\forall\xi\in\mathcal{H}$. The sequences $\{\xi_i\}_{i=1}^{\infty}$ and $\{\eta_i\}_{i=1}^{\infty}$ are called a dual frame pair when $\{\eta_i\}_{i=1}^{\infty}$ is also a frame.
\\Two sequences $\{\xi_i\}_{i=1}^{\infty}$ and $\{\eta_i\}_{i=1}^{\infty}$ are called biorthogonal if $\langle \xi,\eta_{i}\rangle_{\mathcal{A}}=0$, for $i\neq j$ and $\langle \xi,\eta_{i}\rangle_{\mathcal{A}}=1_{\mathcal{A}}$.
\begin{definition}\cite{Koc}
	Consider $C$ as an element of $GL(\mathcal{H})$. A set $\Xi=\{\xi_i\}_{i=1}^{\infty}$ is recognized as a controlled frame in $\mathcal{H}$, or more specifically a $C$-Controlled frame in $\mathcal{H}$, provided there are positive constants $A$ and $B$ such that the following inequality is true for all elements $\xi$ in $\mathcal{H}$:
	$$
	A\langle \xi, \xi\rangle_{\mathcal{A}} \leq \sum_{i=1}^{\infty}\left\langle \xi, \xi_i\right\rangle_{\mathcal{A}}\left\langle C \xi_i, \xi\right\rangle_{\mathcal{A}} \leq B\langle \xi, \xi\rangle_{\mathcal{A}}
	\quad \forall \xi\in \mathcal{H}.$$
\end{definition}

\begin{definition}
	Assuming $C_1, C_2 \in GL(\mathcal{H})$, we define a set $\Xi=\{\xi_i\}_{i=1}^{\infty}$ to be a controlled frame in $\mathcal{H}$, or specifically, a $(C_1, C_2)$-Controlled frame in $\mathcal{H}$. This definition is valid if there are positive constants $A$ and $B$ that satisfy the following inequality for every $\xi \in \mathcal{H}$:
	$$
	A\langle \xi, \xi\rangle_{\mathcal{A}} \leq \sum_{i=1}^{\infty}\left\langle \xi, C_1\xi_i\right\rangle_{\mathcal{A}}\left\langle C_2 \xi_i, \xi\right\rangle_{\mathcal{A}} \leq B\langle \xi, \xi\rangle_{\mathcal{A}}
	\quad \forall \xi\in \mathcal{H}.$$
\end{definition}

We define the frame operator $S_{\Xi}$ in the following manner: $S_{\Xi}:\mathcal{H} \rightarrow \mathcal{H}$ is expressed as $$S_{\Xi}(\xi) = \sum_{i=1}^{\infty}\left\langle \xi, \xi_i\right\rangle_{\mathcal{A}} \xi_{i}$$ for every $\xi \in \mathcal{H}$.
\begin{theorem}\cite{Amr}\label{22122.7}
	Let $I$ be a finite or countable index set. A sequence $\{\xi_i\}_{i \in I}$ is a $K$-Riesz basis for $\mathcal{H}$ if and only if there is a bounded below operator $\Theta \in \operatorname{End}_{\mathcal{A}}^*(\mathcal{H})$ such that $R(K) \subset R(\Theta)$ and $\Theta e_i=\xi_i$ for all $i \in I$, where $\{e_i\}_{i \in I}$ represents a standard orthonormal basis of $\mathcal{H}$.
\end{theorem}
\begin{theorem}\cite{Amr}
	Consider $I$ as either a finite or countable set of indices. A sequence$\{\xi_i\}_{i \in I}$ forms a $K$-Riesz basis in the space $\mathcal{H}$ if, and only if, there exists an operator $\Theta$ that is lower-bounded and part of $\operatorname{End}_{\mathcal{A}}^*(\mathcal{H})$, ensuring that $R(K) \subset R(\Theta)$. Moreover, for every index $i$ in $I$, it is required that $\Theta e_i=\xi_i$, where $\{e_i\}_{i \in I}$ is a standard orthonormal basis in $\mathcal{H}$.
\end{theorem}

\section{The Biframes in Hilbert $C^{\ast}-$modules}

In this section, we introduce the concept of biframes in Hilbert $C^{\ast}$-modules and subsequently present examples illustrating this concept.
\begin{definition}\label{D31}
	A pair $(\Xi, \Upsilon)=\left(\left\{\xi_i\right\}_{i=1}^{\infty},\left\{\eta_i\right\}_{i=1}^{\infty}\right)$ consisting of two sequences in $\mathcal{H}$ is termed a biframe for $\mathcal{H}$ if there are positive constants $A$ and $B$ such that for all $\xi \in \mathcal{H}$, the following inequality is satisfied:
	$$
	A\langle \xi, \xi\rangle_{\mathcal{A}} \leq \sum_{i=1}^{\infty}\left\langle \xi, \xi_i\right\rangle_{\mathcal{A}}\left\langle \eta_i, \xi\right\rangle_{\mathcal{A}} \leq B\langle \xi, \xi\rangle_{\mathcal{A}}.
	$$
The constants $A$ and $B$ are referred to as the lower and upper bounds of the biframe, respectively. A biframe is considered $\lambda$-tight if $A = B = \lambda$. If $A = B = 1$, the pair $(\Xi, \Upsilon)$ is designated as a Parseval biframe.
\end{definition}
Observing Definition \ref{D31}, it becomes evident that biframes extend the ideas of frames and controlled frames. Specifically, for $\Xi = \{\xi_i\}_{i=1}^\infty$, we note:
\begin{enumerate}
	\item If $(\Xi, \Xi)$ constitutes a biframe for $\mathcal{H}$, then $\Xi$ serves as a frame for $\mathcal{H}$.
	\item Given some $C_1$ and $C_2$ in $GL(\mathcal{H})$, if $(C_1 \Xi, C_2 \Xi)$ forms a biframe, then $\Xi$ represents a $(C_1, C_2)$-controlled frame for $\mathcal{H}$.
\end{enumerate}

\begin{example}
	Consider $\{e_k\}_{i=1}^{\infty}$ as an orthonormal basis for $\mathcal{H}$. Define two sequences $\Xi=\{\xi_i\}_{i=1}^{\infty}$ and $\Upsilon=\{\eta_i\}_{i=1}^{\infty}$ as follows:
	$$\Xi=\{e_1, 2 e_2, \frac{1}{3} e_3, 4 e_4, \ldots\},$$
	and
	$$\Upsilon=\{2 e_1, e_2, 4 e_3, \frac{1}{3} e_4, \ldots\}.$$
	Although $\Xi$ and $\Upsilon$ do not form Bessel sequences, the pair $(\Xi, \Upsilon)$ constitutes a biframe. For any $\xi \in \mathcal{H}$, we can express
	\begin{align*}
		\sum_{i=1}^{\infty}\left\langle \xi, \xi_i\right\rangle_{\mathcal{A}}\left\langle \eta_i, \xi\right\rangle_{\mathcal{A}} &= \left\langle \xi, e_1\right\rangle_{\mathcal{A}}\left\langle 2 e_1, \xi\right\rangle_{\mathcal{A}}+ \left\langle \xi, 2 e_2\right\rangle_{\mathcal{A}}\left\langle e_2, \xi\right\rangle_{\mathcal{A}}\\
		&+ \left\langle \xi, \frac{1}{3} e_3\right\rangle_{\mathcal{A}}\left\langle 4 e_3, \xi\right\rangle_{\mathcal{A}} + \left\langle \xi, 4 e_4\right\rangle_{\mathcal{A}}\left\langle \frac{1}{3} e_4, \xi\right\rangle_{\mathcal{A}}\ldots\\
		&= 2\left\langle \xi, e_1\right\rangle_{\mathcal{A}}\left\langle e_1, \xi\right\rangle_{\mathcal{A}} + 2\left\langle \xi, e_2\right\rangle_{\mathcal{A}}\left\langle e_2, \xi\right\rangle_{\mathcal{A}}\\
		&+ \frac{4}{3} \left\langle \xi,e_3\right\rangle_{\mathcal{A}}\left\langle e_3, \xi\right\rangle_{\mathcal{A}} + \frac{4}{3} \left\langle \xi, e_4\right\rangle_{\mathcal{A}}\left\langle e_4, \xi\right\rangle_{\mathcal{A}}\ldots\\
		&= \sum_{i=1}^{\infty} \frac{2i}{2i-1}\left\langle \xi, e_{2i-1}\right\rangle_{\mathcal{A}}\left\langle e_{2i-1},\xi \right\rangle_{\mathcal{A}} + \sum_{i=1}^{\infty} \frac{2i}{2i-1}\left\langle \xi, e_{2i}\right\rangle_{\mathcal{A}}\left\langle e_{2i},\xi \right\rangle_{\mathcal{A}}
	\end{align*}
	we find that
	$$
	\langle \xi, \xi\rangle_{\mathcal{A}} \leq \sum_{i=1}^{\infty}\left\langle \xi, \xi_i\right\rangle_{\mathcal{A}}\left\langle \eta_i, \xi\right\rangle_{\mathcal{A}}\leq 2\langle \xi, \xi\rangle_{\mathcal{A}}
	$$
	This leads to the conclusion that it is feasible to construct a biframe from two sequences that are not Bessel.
\end{example}

\begin{definition}\cite{Aza}
	The pair $(\Xi, \Upsilon)=\left(\left\{\xi_i\right\}_{i=1}^{\infty},\left\{\eta_i\right\}_{i=1}^{\infty}\right)$ is described as an pair frame for $\mathcal{H}$ if the operator $S_{\Xi, \Upsilon}: \mathcal{H} \rightarrow \mathcal{H}$, defined as $$S_{\Xi, \Upsilon}(\xi) := \sum_{i=1}^{\infty}\left\langle \xi, \xi_i\right\rangle_{\mathcal{A}} \eta_i$$ is well-defined, possesses an adjoint, and is invertible.
\end{definition}

\begin{example}
	Consider the sequences $\Xi=\left\{\xi_i\right\}_{i=1}$ and $\Upsilon=\left\{\eta_i\right\}_{i=1}$ defined as:
	$$
	\Xi=\left\{(1,2), (3, 4)\right\},
	$$
	and
	$$
	\Upsilon=\left\{(1, 1), (1, -1)\right\}.
	$$
	
	Analyzing the properties of the pair frame operator $S$. For  $(\xi, \eta) \in \mathbb{R}^2$, $S$ is defined by:
	$$
	\begin{aligned}
		S(\xi, \eta) & = \langle(\xi, \eta),(1,2)\rangle_{\mathcal{A}}(1, 1) + \langle(\xi, \eta), (3, 4)\rangle_{\mathcal{A}}(1, -1) \\
		& =(\xi+2 \eta)(1, 1) + (3 \xi+4 \eta)(1, -1) \\
		& =(\xi+2 \eta, \xi+2 \eta) + (3 \xi+4 \eta, -3\xi-4 \eta) \\
		& =(4\xi+6 \eta, -2 \xi-2 \eta).
	\end{aligned}
	$$
	
	The matrix representing the operator $S$ is
	$$
	S=\left(\begin{array}{cc}
		4 & 6 \\
		-2 & -2
	\end{array}\right).
	$$
	
	Since $\operatorname{det}(S)=-8+12=4 \neq 0$, the matrix $S$ is invertible. Therefore, the operator $S$ is well-defined and invertible, making $(\Xi, \Upsilon)$ a pair frame. However, this pair does not constitute a biframe. For instance, for $(\xi, \eta) \in \mathbb{R}^2$,
	$$
	\langle(\xi, \eta),(1,2)\rangle_{\mathcal{A}}\langle(1, 1),(\xi, \eta)\rangle_{\mathcal{A}}+\langle(\xi, \eta), (3, 4)\rangle_{\mathcal{A}}\langle(1, -1),(\xi, \eta)\rangle_{\mathcal{A}} = 4\xi^2-2\eta^2+4\xi \eta.
	$$
	
	Setting $\xi=\frac{-1+\sqrt{3}}{2}$ and $\eta=1$, we find $4\xi^2-2\eta^2+4\xi \eta = 0$. Hence, $\sum_{i=1}^{2}\langle \xi, \xi_i\rangle_{\mathcal{A}}\langle \eta_i, \xi\rangle_{\mathcal{A}}$ lacks a nonzero lower bound.
\end{example}

\section{The biframe operator}

Determining the bounds of a biframe can often be challenging in practice. Therefore, it is essential to introduce an operator akin to the frame operator, possessing equally beneficial and practical qualities. 

In this section, we explore various properties of the biframe operator associated with a biframe. Additionally, we utilize the characteristics of this operator to define biframes.

Consider the pair $(\Xi, \Upsilon)=\left(\left\{\xi_i\right\}_{i=1}^{\infty},\left\{\eta_i\right\}_{i=1}^{\infty}\right)$ as a biframe for $\mathcal{H}$. The biframe operator $S_{\Xi, \Upsilon}$ is defined in the following manner:
$$
S_{\Xi, \Upsilon}: \mathcal{H} \rightarrow \mathcal{H}, \quad S_{\Xi, \Upsilon}(\xi) := \sum_{i=1}^{\infty}\left\langle \xi, \xi_i\right\rangle_{\mathcal{A}} \eta_i .
$$
\begin{theorem}\label{T1}
	Assume that the pair $(\Xi, \Upsilon)=\left(\left\{\xi_i\right\}_{i=1}^{\infty},\left\{\eta_i\right\}_{i=1}^{\infty}\right)$ forms a biframe for $\mathcal{H}$, characterized by bounds $A$ and $B$. Under these conditions, the following assertions hold true:
	\begin{enumerate}
		\item The operator $S_{\Xi, \Upsilon}$ is well-defined, self-adjoint, bounded, positive, and invertible.
		\item The pair $(\Xi, \Upsilon)$ constitutes a biframe if and only if the pair $(\Upsilon, \Xi)$ also forms a biframe.
	\end{enumerate}
\end{theorem}

\begin{proof}

	(1) For any element $\xi \in \mathcal{H}$ and a natural number $n$, consider the definition
	$$
	S_n \xi = \sum_{i=1}^n\left\langle \xi, \xi_i\right\rangle_{\mathcal{A}} \eta_i .
	$$
	
	The set $\{S_n\}_{n=1}^{\infty}$ forms a sequence of linear and bounded operators on $\mathcal{H}$. Given $l, m \in \mathbb{N}$ with $l > m$, we have
	$$
	\begin{aligned}
		S_l \xi - S_m \xi &= \sum_{i=1}^l\left\langle \xi, \xi_i\right\rangle_{\mathcal{A}} \eta_i - \sum_{i=1}^m\left\langle \xi, \xi_i\right\rangle_{\mathcal{A}} \eta_i \\
		&= \sum_{i=m+1}^l\left\langle \xi, \xi_i\right\rangle_{\mathcal{A}} \eta_i .
	\end{aligned}
	$$
	Hence,
	$$
	\begin{aligned}
		\Vert\left\langle S_l \xi, \xi\right\rangle_{\mathcal{A}} - \left\langle S_m \xi, \xi\right\rangle_{\mathcal{A}}\Vert &= \Vert\left\langle(S_l \xi - S_m \xi), \xi\right\rangle_{\mathcal{A}}\Vert \\
		&= \left|\left|\left\langle \sum_{i=m+1}^l\left\langle \xi, \xi_i\right\rangle_{\mathcal{A}} \eta_i, \xi\right\rangle_{\mathcal{A}}\right|\right| \\
		&= \left|\left| \sum_{i=m+1}^l\left\langle \xi, \xi_i\right\rangle_{\mathcal{A}}\left\langle \eta_i, \xi\right\rangle_{\mathcal{A}}\right|\right| .
	\end{aligned}
	$$
	
	The series $\sum_{i=1}^{\infty}\left\langle \xi, \xi_i\right\rangle_{\mathcal{A}}\left\langle \eta_i, \xi\right\rangle_{\mathcal{A}}$ converges in norm. Therefore, the sequence $\{\left\langle S_n \xi, \xi\right\rangle_{\mathcal{A}}\}_{n=1}^{\infty}$ is Cauchy, leading to the existence of a bounded operator $S$ to which $\{S_n\}_{n=1}^{\infty}$ converges weakly:
	$$
	\begin{aligned}
		\langle S \xi, \xi\rangle_{\mathcal{A}} &= \lim_{n \to \infty} \left\langle S_n \xi, \xi\right\rangle_{\mathcal{A}} \\
		&= \left\langle \sum_{i=1}^{\infty}\left\langle \xi, \xi_i\right\rangle_{\mathcal{A}} \eta_i, \xi\right\rangle_{\mathcal{A}} \\
		&= \left\langle S_{\Xi, \Upsilon} \xi, \xi\right\rangle_{\mathcal{A}}.
	\end{aligned}
	$$
	
	By the definition of the biframe operator $S_{\Xi, \Upsilon}$ and the uniqueness of the limit, we conclude that $$S_{\Xi, \Upsilon} = S.$$ Therefore, $S_{\Xi, \Upsilon}$ is a well-defined and bounded operator.

	For each element $\xi \in \mathcal{H}$, the following holds true:
	$$\sum_{i=1}^{\infty}\left\langle \xi, \xi_i\right\rangle_{\mathcal{A}}\left\langle \eta_i, \xi\right\rangle_{\mathcal{A}}=\left\langle S_{\Xi, \Upsilon} \xi, \xi\right\rangle_{\mathcal{A}}.$$
	As per Definition \ref{D31}, we deduce
	\begin{equation}\label{Eq1}
		A\langle \xi, \xi\rangle_{\mathcal{A}} \leq \sum_{i=1}^{\infty}\left\langle \xi, \xi_i\right\rangle_{\mathcal{A}}\left\langle \eta_i, \xi\right\rangle_{\mathcal{A}}=\left\langle S_{\Xi, \Upsilon} \xi, \xi\right\rangle_{\mathcal{A}} \leq B\langle \xi, \xi\rangle_{\mathcal{A}},
	\end{equation}
	indicating that $S_{\Xi, \Upsilon}$ is a positive operator.
	
	For any $\xi, \eta \in \mathcal{H}$, we have
	$$
	\left\langle S_{\Xi, \Upsilon} \xi, \eta\right\rangle_{\mathcal{A}} = \sum_{i=1}^{\infty}\left\langle \xi, \xi_i\right\rangle_{\mathcal{A}}\left\langle \eta_i, \eta\right\rangle_{\mathcal{A}} = \left\langle \xi, S_{\Upsilon, \Xi} \eta\right\rangle_{\mathcal{A}},
	$$
	implying that $S_{\Xi, \Upsilon}^* = S_{\Upsilon, \Xi}$. According to the definition of a biframe, both $S_{\Xi, \Upsilon}$ and $S_{\Upsilon, \Xi}$ are injective.
	
	Now, consider a sequence $\{\alpha_n\}_{n=1}^{\infty} \subset \mathcal{R}(S_{\Xi, \Upsilon})$ that converges to some $\alpha \in \mathcal{H}$. Then, there exists a sequence $\{\beta_n\}_{n=1}^{\infty}$ in $\mathcal{H}$ such that $S_{\Xi, \Upsilon}(\beta_n)=\alpha_n$ for all $n$, meaning the sequence $\{S_{\Xi, \Upsilon}(\beta_n)\}_{n=1}^{\infty}$ converges to $\alpha$. As the sequence $\{\alpha_n\}_{n=1}^{\infty}$ is convergent, it is also a Cauchy sequence.

	Given any $\varepsilon > 0$, there exists an $N > 0$ such that
	$$
	\left\|S_{\Xi, \Upsilon}\left(\beta_n\right) - S_{\Xi, \Upsilon}\left(\beta_m\right)\right\| = \left\|\alpha_n - \alpha_m\right\| \leq \varepsilon, \quad \forall m, n \geq N.
	$$
	For $m, n \geq N$ with $n > m$, and for any $\xi \in \mathcal{H}$, by (\ref{Eq1}), we have
	$$
	A\langle \xi, \xi\rangle_{\mathcal{A}} \leq \left\langle S_{\Xi, \Upsilon} \xi, \xi\right\rangle_{\mathcal{A}}.
	$$
	Applying the Cauchy-Schwarz inequality, we get
	$$
	\begin{aligned}
		A\left\langle \beta_n - \beta_m, \beta_n - \beta_m\right\rangle_{\mathcal{A}} & \leq \left\langle S_{\Xi, \Upsilon}\left(\beta_n - \beta_m\right), \left(\beta_n - \beta_m\right)\right\rangle_{\mathcal{A}} \\
		& \leq \left\|S_{\Xi, \Upsilon}\left(\beta_n\right) - S_{\Xi, \Upsilon}\left(\beta_m\right)\right\| \left\|\beta_n - \beta_m\right\| \\
		& \leq \varepsilon \left\|\beta_n - \beta_m\right\|.
	\end{aligned}
	$$
	
	Thus, $$\left\langle \beta_n - \beta_m, \beta_n - \beta_m\right\rangle_{\mathcal{A}} \leq \left(\frac{\varepsilon}{A}\right)^{2}.$$ This implies that the sequence $\{\beta_n\}_{n=1}^{\infty}$ is a Cauchy sequence in $\mathcal{H}$ and thus converges to some $\beta \in \mathcal{H}$. As $S_{\Xi, \Upsilon}$ is bounded, we have
	$$
	S_{\Xi, \Upsilon}\left(\beta_n\right) \to S_{\Xi, \Upsilon}(\beta), \quad \text{as } n \to \infty.
	$$
	
	Conversely, we also have
	$$
	S_{\Xi, \Upsilon}\left(\beta_n\right) \to \alpha, \quad \text{as } n \to \infty.
	$$
	
	By the uniqueness of the limit, it follows that $S_{\Xi, \Upsilon}(\beta) = \alpha$. Consequently, $\alpha \in \mathcal{R}(S_{\Xi, \Upsilon})$, showing that $\mathcal{R}(S_{\Xi, \Upsilon})$ is closed. Thus, $\mathcal{R}(S_{\Xi, \Upsilon}^*)$ is also closed.
	
	(2) If $(\Xi, \Upsilon)$ is a biframe with bounds $A$ and $B$, then for every $\xi \in \mathcal{H}$,
	$$
	A\langle \xi, \xi\rangle_{\mathcal{A}} \leq \sum_{i=1}^{\infty}\left\langle \xi, \xi_i\right\rangle_{\mathcal{A}}\left\langle \eta_i, \xi\right\rangle_{\mathcal{A}} \leq B\langle \xi, \xi\rangle_{\mathcal{A}}.
	$$
	
	By manipulating the above inequality, we obtain
	$$
	\begin{aligned}
		\sum_{i=1}^{\infty}\left\langle \xi, \xi_i\right\rangle_{\mathcal{A}}\left\langle \eta_i, \xi\right\rangle_{\mathcal{A}} &= \overline{\sum_{i=1}^{\infty}\left\langle \xi, \xi_i\right\rangle_{\mathcal{A}}\left\langle \eta_i, \xi\right\rangle_{\mathcal{A}}} \\
		&= \sum_{i=1}^{\infty}\overline{\left\langle \xi, \xi_i\right\rangle_{\mathcal{A}}\left\langle \eta_i, \xi\right\rangle_{\mathcal{A}}} \\
		&= \sum_{i=1}^{\infty}\left\langle \xi, \eta_i\right\rangle_{\mathcal{A}}\left\langle \xi_i, \xi\right\rangle_{\mathcal{A}}.
	\end{aligned}
	$$
	Hence, we can conclude that
	$$
	A\langle \xi, \xi\rangle_{\mathcal{A}} \leq \sum_{i=1}^{\infty}\left\langle \xi, \eta_i\right\rangle_{\mathcal{A}}\left\langle \xi_i, \xi\right\rangle_{\mathcal{A}} \leq B\langle \xi, \xi\rangle_{\mathcal{A}}.
	$$
	
	This indicates that $(\Xi, \Upsilon)$ is a biframe with bounds $A$ and $B$, and a similar argument can be made for the converse.
	
\end{proof}

\begin{theorem}
	Assume that the pair $(\Xi, \Upsilon)=\left(\left\{\xi_i\right\}_{i=1}^{\infty},\left\{\eta_i\right\}_{i=1}^{\infty}\right)$ forms a biframe for $\mathcal{H}$ with the associated biframe operator $S_{\Xi, \Upsilon}$. Then, for every element $\xi \in \mathcal{H}$, the following reconstruction formula is valid:
	\begin{equation}\label{reconst}
		\xi = \sum_{i=1}^{\infty}\left\langle \xi, S_{\Upsilon, \Xi}^{-1} \xi_i\right\rangle_{\mathcal{A}} \eta_i = \sum_{i=1}^{\infty}\left\langle \xi, \xi_i\right\rangle_{\mathcal{A}} S_{\Xi, \Upsilon}^{-1} \eta_i .
	\end{equation}
\end{theorem}

\begin{proof}
	Referencing Theorem \ref{T1}, we recall that $S_{\Xi, \Upsilon}^* = S_{\Upsilon, \Xi}$. For any element $\xi \in \mathcal{H}$, we have
	$$
	\xi = S_{\Xi, \Upsilon} S_{\Xi, \Upsilon}^{-1} \xi = \sum_{i=1}^{\infty}\left\langle S_{\Xi, \Upsilon}^{-1}\xi, \xi_i\right\rangle_{\mathcal{A}} \eta_i = \sum_{i=1}^{\infty}\left\langle \xi, S_{\Upsilon, \Xi}^{-1}\xi_i\right\rangle_{\mathcal{A}} \eta_i
	$$
	and similarly,
	$$
	\xi = S_{\Xi, \Upsilon}^{-1} S_{\Xi, \Upsilon} \xi = S_{\Xi, \Upsilon}^{-1} \sum_{i=1}^{\infty}\left\langle \xi, \xi_i\right\rangle_{\mathcal{A}} \eta_i = \sum_{i=1}^{\infty}\left\langle \xi, \xi_i\right\rangle_{\mathcal{A}} S_{\Xi, \Upsilon}^{-1} \eta_i.
	$$
\end{proof}

\begin{theorem}\label{T43}
	Consider sequences $\Xi=\left\{\xi_i\right\}_{i=1}^{\infty}$ and $\Upsilon=\left\{\eta_i\right\}_{i=1}^{\infty}$ within a Hilbert $C^{\ast}-$module $\mathcal{H}$, equipped with a biframe operator $S_{\Xi, \Upsilon}$ that fulfills the condition:
	\begin{equation}\label{eq43} 
		\left\langle S_{\Xi, \Upsilon} \xi, S_{\Xi, \Upsilon} \xi\right\rangle_{\mathcal{A}} \leq \left\|S_{\Xi, \Upsilon}\right\| \left\langle S_{\Xi, \Upsilon} \xi, \xi\right\rangle_{\mathcal{A}} \quad (\forall \xi \in \mathcal{H}).
	\end{equation}
	Under these circumstances, the following statements are equivalent:
	\begin{enumerate}
		\item The pair $(\Xi, \Upsilon)$ constitutes a biframe for $\mathcal{H}$.
		\item The biframe operator $S_{\Xi, \Upsilon}$ is positive and bounded below.
	\end{enumerate}
\end{theorem}

\begin{proof}
	(1)$\Rightarrow$(2) To establish this implication, we refer to Theorem \ref{T1}, where $S_{\Xi, \Upsilon}$ is confirmed as a positive and invertible operator. Hence, it is necessarily bounded below. 
	
	(2)$\Rightarrow$(1) Suppose $S_{\Xi, \Upsilon}$ is a positive and bounded below operator on $\mathcal{H}$. Given its positive nature and its in $\operatorname{End}_\mathcal{A}^*(\mathcal{H})$, for any $\xi \in \mathcal{H}$, we have
	$$
	\sum_{i=1}^{\infty}\left\langle \xi, \xi_i\right\rangle_{\mathcal{A}}\left\langle \eta_i, \xi\right\rangle_{\mathcal{A}}=\left\langle S_{\Xi, \Upsilon} \xi, \xi\right\rangle_{\mathcal{A}} \leq \left\|S_{\Xi, \Upsilon}\right\|\langle \xi, \xi\rangle_{\mathcal{A}} .
	$$
	
	To establish the lower bound, the hypothesis (\ref{eq43}) of this theorem provides:
	\begin{equation}\label{eq4.5}
		\left\langle S_{\Xi, \Upsilon} \xi, S_{\Xi, \Upsilon} \xi\right\rangle_{\mathcal{A}} \leq \left\|S_{\Xi, \Upsilon}\right\|\left\langle S_{\Xi, \Upsilon} \xi, \xi\right\rangle_{\mathcal{A}} .
	\end{equation}
	
	Additionally, since $S_{\Xi, \Upsilon}$ is bounded below with a lower bound $\alpha$, we have:
	\begin{equation}\label{eq4.6}
		\exists \alpha > 0, \quad \alpha\|\xi\| \leq \left\|S_{\Xi, \Upsilon} \xi\right\| .
	\end{equation}
	
	Utilizing both \ref{eq4.5} and \ref{eq4.6}, we infer that:
	$$
	\alpha^2\langle \xi, \xi\rangle_{\mathcal{A}} \leq \left\|S_{\Xi, \Upsilon}\right\|\left\langle S_{\Xi, \Upsilon} \xi, \xi\right\rangle_{\mathcal{A}}, \quad \forall \xi \in \mathcal{H} .
	$$
	
	This leads to the conclusion that $\left\langle S_{\Xi, \Upsilon} \xi, \xi\right\rangle_{\mathcal{A}} \geq \alpha^2\left\|S_{\Xi, \Upsilon}\right\|^{-1}\langle \xi, \xi\rangle_{\mathcal{A}}$, suggesting that $\alpha^2\left\|S_{\Xi, \Upsilon}\right\|^{-1}$ is a lower bound for $\left\langle S_{\Xi, \Upsilon} \xi, \xi\right\rangle_{\mathcal{A}}$. Consequently, the pair $(\{\xi_i\}_{i=1}^{\infty}, \{\eta_i\}_{i=1}^{\infty})$ qualifies as a biframe for $\mathcal{H}$.
\end{proof}

To address the inquiries raised in the preceding section regarding the structural relationship between two sequences that constitute a biframe, we present the following theorem. This theorem elucidates the dependence of two Riesz bases that create a biframe. Initially, we demonstrate through various examples that such dependence is not a given for Bessel sequences and frames.
\begin{example}
	Let $\{e_k\}_{i=1}^{\infty}$ be an orthonormal basis for $\mathcal{H}$. Consider sequences $\Xi=\{\xi_i\}_{i=1}$ and $\Upsilon=\{\eta_i\}_{i=1}$ defined as follows:
	
	$$\Xi=\{e_1, e_1, e_1, e_2, e_2, e_2, e_3, e_3, e_3, \ldots\},$$ and
	$$\Upsilon=\{2 e_1, e_1, -e_1, \frac{3}{2} e_2, e_2, -e_2, \frac{4}{3} e_3, e_3, -e_3, \ldots\}.$$
	
	For $\xi \in \mathcal{H}$, the following computation holds:
	\begin{align*}
		\sum_{i=1}^{\infty}\left\langle \xi, \xi_i\right\rangle_{\mathcal{A}}\left\langle \eta_i, \xi\right\rangle_{\mathcal{A}} 
		&= \left\langle \xi, e_1\right\rangle_{\mathcal{A}}\left\langle 2e_1, \xi\right\rangle_{\mathcal{A}} 
		+ \left\langle \xi, e_1\right\rangle_{\mathcal{A}}\left\langle e_1, \xi\right\rangle_{\mathcal{A}} \\
		&- \left\langle \xi, e_1\right\rangle_{\mathcal{A}}\left\langle e_1, \xi\right\rangle_{\mathcal{A}}  + \left\langle \xi, e_2\right\rangle_{\mathcal{A}}\left\langle \frac{3}{2}e_2, \xi\right\rangle_{\mathcal{A}} + \cdots \\
		&= 2\left\langle \xi, e_1\right\rangle_{\mathcal{A}}\left\langle e_1, \xi\right\rangle_{\mathcal{A}} 
		+ \frac{3}{2}\left\langle \xi, e_2\right\rangle_{\mathcal{A}}\left\langle e_2, \xi\right\rangle_{\mathcal{A}} 
		+ \frac{4}{3}\left\langle \xi, e_3\right\rangle_{\mathcal{A}}\left\langle e_3, \xi\right\rangle_{\mathcal{A}} + \cdots \\
		&= \sum_{i=1}^{\infty} \frac{i+1}{i}\left\langle \xi, e_i\right\rangle_{\mathcal{A}}\left\langle e_i, \xi\right\rangle_{\mathcal{A}}.
	\end{align*}
	it follows that
	$$
	\langle \xi, \xi\rangle_{\mathcal{A}} \leq \sum_{i=1}^{\infty}\left\langle \xi, \xi_i\right\rangle_{\mathcal{A}}\left\langle \eta_i, \xi\right\rangle_{\mathcal{A}} \leq 2\langle \xi, \xi\rangle_{\mathcal{A}}.
	$$
	Consequently, the pair $(\Xi, \Upsilon)$ forms a biframe with bounds 1 and 2, despite $\{\xi_i\}_{i=1}^{\infty}$ being a frame and $\{\eta_i\}_{i=1}^{\infty}$ not being a frame. This example illustrates that a biframe can be composed of a frame and a non-frame.
\end{example}

\begin{example}
	Consider an orthonormal basis $\{e_k\}_{i=1}^{\infty}$ for a Hilbert space $\mathcal{H}$. Let's define two sequences within this space, namely $\Xi=\{\xi_i\}_{i=1}$ and $\Upsilon=\{\eta_i\}_{i=1}$, as follows:
	
	$$\Xi=\left\{e_1, \frac{1}{\sqrt{2}} e_2, e_3, \frac{1}{2} e_4, e_5, \frac{1}{\sqrt{6}} e_6, \ldots\right\},$$ 
	and 
	$$\Upsilon=\left\{e_1, \sqrt{2} e_2, e_3, 2 e_4, e_5, \sqrt{6} e_6, \ldots\right\}.$$
		For an arbitrary element $\xi$ in $\mathcal{H}$, we can derive the following series of equalities:
	
	\begin{align*}
		\sum_{i=1}^{\infty}\left\langle \xi, \xi_i\right\rangle_{\mathcal{A}}\left\langle \eta_i, \xi\right\rangle_{\mathcal{A}}&= \left\langle \xi, e_1\right\rangle_{\mathcal{A}}\left\langle e_1, \xi\right\rangle_{\mathcal{A}}+ \left\langle \xi, \frac{1}{\sqrt{2}} e_2\right\rangle_{\mathcal{A}}\left\langle \sqrt{2} e_2, \xi\right\rangle_{\mathcal{A}} \\
		&+ \left\langle \xi, e_3\right\rangle_{\mathcal{A}}\left\langle e_3, \xi\right\rangle_{\mathcal{A}} + \left\langle \xi, \frac{1}{2} e_4\right\rangle_{\mathcal{A}}\left\langle 2 e_4, \xi\right\rangle_{\mathcal{A}} \\
		&+ \left\langle \xi, e_5\right\rangle_{\mathcal{A}}\left\langle e_5, \xi\right\rangle_{\mathcal{A}} + \left\langle \xi, \frac{1}{\sqrt{6}} e_6\right\rangle_{\mathcal{A}}\left\langle \sqrt{6} e_6, \xi\right\rangle_{\mathcal{A}} + \cdots \\
		&= \sum_{i=1}^{\infty} \left\langle \xi, e_i\right\rangle_{\mathcal{A}}\left\langle e_i, \xi\right\rangle_{\mathcal{A}}.
	\end{align*}
	
	As highlighted in Example 2.1 in \cite{Xia}, this results in
	$$
	\sum_{i=1}^{\infty}\left\langle \xi, \xi_i\right\rangle_{\mathcal{A}}\left\langle \eta_i, \xi\right\rangle_{\mathcal{A}} = \langle \xi, \xi\rangle_{\mathcal{A}}.
	$$
	
	Therefore, $ ( \Xi,\Upsilon ) $ forms a Parseval biframe for $\mathcal{H}$, even though $\Xi$ is recognized as a Bessel sequence while $\Upsilon$ is not. This scenario illustrates that a Parseval biframe can indeed be composed of a Bessel sequence paired with a non-Bessel sequence.
\end{example}

\begin{theorem}
	Assume that the pair $(\{\xi_i\}_{i=1}^{\infty}, \{\eta_i\}_{i=1}^{\infty})$ constitutes a biframe for $\mathcal{H}$. In this context, the sequence $\{\xi_i\}_{i=1}^{\infty}$ forms a Riesz basis for $\mathcal{H}$ if and only if the sequence $\{\eta_i\}_{i=1}^{\infty}$ also forms a Riesz basis for $\mathcal{H}$.
\end{theorem}

\begin{proof}
	Assume $\{\xi_i\}_{i=1}^{\infty}$ is a Riesz basis. Utilizing Theorem \ref{22122.7} (specifically for the case when $K$ is the identity operator), there exists an operator $\Theta \in \operatorname{End}_{\mathcal{A}}^*(\mathcal{H})$ such that $\xi_i= \Theta e_i$ for each $i \in \mathbb{N}$, where $\{e_i\}_{i=1}^{\infty}$ is an orthonormal basis of $\mathcal{H}$.
	
	Now, define the operator
	$$
	U: \mathcal{H} \rightarrow \mathcal{H}, \quad U \xi=S_{\Xi, \Upsilon}(\Theta^*)^{-1} \xi.
	$$
	For any $j\in \mathbb{N}$, we observe
	$$
	\begin{aligned}
		U e_j &= S_{\Xi, \Upsilon}(\Theta^*)^{-1} e_j \\
		&= \sum_{i=1}^{\infty}\left\langle(\Theta^*)^{-1} e_j, \xi_i\right\rangle_{\mathcal{A}} \eta_i \\
		&= \sum_{i=1}^{\infty}\left\langle(\Theta^*)^{-1} e_j, \Theta e_i\right\rangle_{\mathcal{A}} \eta_i \\
		&= \sum_{i=1}^{\infty}\left\langle \Theta^*(\Theta^*)^{-1} e_j, e_i\right\rangle_{\mathcal{A}} \eta_i \\
		&= \sum_{i=1}^{\infty}\left\langle e_j, e_i\right\rangle_{\mathcal{A}} \eta_i \\
		&= \eta_j.
	\end{aligned}
	$$
	
	Since $U \in \operatorname{End}_{\mathcal{A}}^*(\mathcal{H})$, it follows that $\eta_i = U e_k$ for all $k \in \mathbb{N}$. Therefore, $\{\eta_i\}_{i=1}^{\infty}$ is a Riesz basis for $\mathcal{H}$. Similarly, it can be shown that $\{\xi_i\}_{i=1}^{\infty}$ is a Riesz basis if $\{\eta_i\}_{i=1}^{\infty}$ is a Riesz basis.
\end{proof}

\begin{lemma}\label{L46}
	Suppose $\mathcal{T}_1, \mathcal{T}_2 \in GL^{+}(\mathcal{H})$. Hence $U \mathcal{T}_1 S^* = \mathcal{T}_2$ holds if and only if $S = \mathcal{T}_2^s P \mathcal{T}_1^{-q} \in \operatorname{End}_\mathcal{A}^*(\mathcal{H})$ and $U = \mathcal{T}_2^r Q \mathcal{T}_1^{-p} \in \operatorname{End}_\mathcal{A}^*(\mathcal{H})$, where $Q, P \in \operatorname{End}_\mathcal{A}^*(\mathcal{H})$ and $p, q, s, r \in \mathbb{R}$ such that $QP^* = I d_\mathcal{H}$ and $p+q=1, s+r=1$.
\end{lemma}

\begin{proof}
	Consider real numbers $p, q, r, s$ such that $p+q=1$ and $r+s=1$. Let $\mathcal{T}_1, \mathcal{T}_2 \in GL^{+}(\mathcal{H})$ and suppose $U \mathcal{T}_1 S^*=\mathcal{T}_2$ where $S, U \in \operatorname{End}_\mathcal{A}^*(\mathcal{H})$. Then, we have:
	$$
	\begin{aligned}
		\mathcal{T}_2^{-r} U \mathcal{T}_1^p \mathcal{T}_1^q S^* \mathcal{T}_2^{-s} &= \mathcal{T}_2^{-r} U \mathcal{T}_1^{(p+q)} S^* \mathcal{T}_2^{-s} \\
		&= \mathcal{T}_2^{(-r+1-s)} \\
		&= I d_\mathcal{H}.
	\end{aligned}
	$$
	
	Now, define $Q:=\mathcal{T}_2^{-r} U \mathcal{T}_1^p$ and $P:=\mathcal{T}_2^{-s} S \mathcal{T}_1^q$. We can observe that $Q, P \in \operatorname{End}_\mathcal{A}^*(\mathcal{H})$ and fulfill the condition $Q P^* = I d_H$. Furthermore, it follows that $S=\mathcal{T}_2^s P \mathcal{T}_1^{-q}$ and $U=\mathcal{T}_2^r Q \mathcal{T}_1^{-p}$.
	
	Conversely, if $P, Q \in \operatorname{End}_\mathcal{A}^*(\mathcal{H})$ with $Q P^* = I d_H$, we can define $S:=\mathcal{T}_2^s P \mathcal{T}_1^{-q}$ and $U:=\mathcal{T}_2^r Q \mathcal{T}_1^{-p}$. It is straightforward to verify that $S, U \in \operatorname{End}_\mathcal{A}^*(\mathcal{H})$ and satisfy the condition $U \mathcal{T}_1 S^* = \mathcal{T}_2$.
\end{proof}

\begin{theorem}
	Given a biframe $(\Xi, \Upsilon) = \left(\{\xi_i\}_{i=1}^{\infty}, \{\eta_i\}_{i=1}^{\infty}\right)$ in $\mathcal{H}$, with the designated biframe operator $S_{\Xi, \Upsilon}$,then the pair $\left(\{P \xi_i\}_{i=1}^{\infty}, \{Q \eta_i\}_{i=1}^{\infty}\right)$ also forms a biframe for $\mathcal{H}$ for some operators $P$ and $Q$ in $\operatorname{End}_\mathcal{A}^*(\mathcal{H})$. This equivalence is true if and only if there are operators $R$, which is positive and bounded from below, and $M, N \in \operatorname{End}_\mathcal{A}^*(\mathcal{H})$ satisfying the condition $M N^* = I$. The operators $P$ and $Q$ must be expressible as $P=R^r N S_{\Xi, \Upsilon}^{-p}$ and $Q=R^t M S_{\Xi, \Upsilon}^{-q}$, for real numbers $p, q, r, t$ that meet the criteria $p+q=1$ and $r+t=1$.
\end{theorem}

\begin{proof}
	Assume that the pair $(P\Xi, Q\Upsilon) = \left(\{P \xi_i\}_{i=1}^{\infty}, \{Q \eta_i\}_{i=1}^{\infty}\right)$ forms a biframe for the Hilbert space $\mathcal{H}$, with the associated biframe operator $S_{P\Xi, Q\Upsilon}$ and where $P, Q$ are operators in $\operatorname{End}_\mathcal{A}^*(\mathcal{H})$. For any element $\xi$ in $\mathcal{H}$, we can express:
	$$
	\begin{aligned}
		S_{P\Xi, Q\Upsilon} \xi &= \sum_{i=1}^{\infty}\left\langle \xi, P \xi_i\right\rangle_{\mathcal{A}} Q \eta_i \\
		&= Q \sum_{i=1}^{\infty}\left\langle P^* \xi, \xi_i\right\rangle_{\mathcal{A}} \eta_i \\
		&= Q S_{\Xi, \Upsilon} P^* \xi.
	\end{aligned}
	$$
	
	Given that both $(\Xi, \Upsilon)$ and $(P\Xi, Q\Upsilon)$ are biframes, according to Theorem \ref{T43}, it follows that the operators $S_{\Xi, \Upsilon}$ and $S_{P\Xi, Q\Upsilon}$ are positive and bounded below, and are also elements of $GL^{+}(\mathcal{H})$.
	
	Invoking Lemma \ref{L46} and the equation $S_{P\Xi, Q\Upsilon} \xi = Q S_{\Xi, \Upsilon} P^* \xi$, we deduce the existence of operators $M, N$ in $\operatorname{End}_\mathcal{A}^*(\mathcal{H})$ that satisfy $M N^* = I$, and thus
	$$
	P = S_{P\Xi, Q\Upsilon}^r N S_{\Xi, \Upsilon}^{-p}, \quad Q = S_{P\Xi, Q\Upsilon}^t M S_{\Xi, \Upsilon}^{-q}.
	$$
	
	On the flip side, consider $R$ as a positive and bounded below operator, with $M$ and $N$ being operators in $\operatorname{End}_\mathcal{A}^*(\mathcal{H})$ such that $M N^* = I$, and define
	$$
	P = R^r N S_{\Xi, \Upsilon}^{-p}, \quad Q = R^t M S_{\Xi, \Upsilon}^{-q}.
	$$
	
	By application of Lemma \ref{L46}, we have $R = Q S_{\Xi, \Upsilon} P^*$. Consequently, for every $\xi$ in $\mathcal{H}$,
	$$
	\begin{aligned}
		R \xi &= Q S_{\Xi, \Upsilon} P^* \xi \\
		&= Q\left(\sum_{i=1}^{\infty}\left\langle P^* \xi, \xi_i\right\rangle_{\mathcal{A}} \eta_i\right) \\
		&= \sum_{i=1}^{\infty}\left\langle \xi, P \xi_i\right\rangle_{\mathcal{A}} Q \eta_i.
	\end{aligned}
	$$
	
	Thus, $R$ takes the form of a biframe operator. Given that $R$ is both positive and bounded below, as per Theorem \ref{T43}, it is concluded that $(P\Xi, Q\Upsilon)$ indeed constitutes a biframe.
\end{proof}

\begin{theorem}\label{T48}
	Suppose that $(\Xi, \Upsilon)=\left(\left\{\xi_i\right\}_{i=1}^{\infty},\left\{\eta_i\right\}_{i=1}^{\infty}\right)$ is a biframe for $\mathcal{H}$ with the biframe operator $S_{\Xi, \Upsilon}$. Then $\left(\left\{U \xi_i\right\}_{i=1}^{\infty},\left\{Q \eta_i\right\}_{i=1}^{\infty}\right)$ constitutes a Parseval biframe for $\mathcal{H}$, for some operators $U$ and $Q$ in $\operatorname{End}_\mathcal{A}^*(\mathcal{H})$, if and only if there exist operators $T, P \in \operatorname{End}_\mathcal{A}^*(\mathcal{H})$ satisfying $T P^*=I$ and $U=P S_{\Xi, \Upsilon}^{-p}, Q=T S_{\Xi, \Upsilon}^{-q}$, with $p, q \in \mathbb{R}$ such that $p+q=1$.
	
	In particular, if $(\left\{\xi_i\right\}_{i=1}^{\infty},\left\{\eta_i\right\}_{i=1}^{\infty})$ is a Parseval biframe, then $(\left\{U \xi_i\right\}_{i=1}^{\infty},\left\{Q \eta_i\right\}_{i=1}^{\infty})$ is a Parseval biframe if and only if $Q U^*=I$.
\end{theorem}

\begin{proof}
	The proof of this theorem is a direct consequence of the previous theorem.
\end{proof}

\begin{corollary}
	Let $P, Q$ be operators in $\operatorname{End}_\mathcal{A}^*(\mathcal{H})$. Considering the sequences $\Xi = \{\xi_i\}_{i=1}^{\infty}$ and $\Upsilon = \{\eta_i\}_{i=1}^{\infty}$ within  $\mathcal{H}$, we have the following conditions:
	\begin{enumerate}
		\item If $\Xi$ and $\Upsilon$ serve as dual frames in $\mathcal{H}$.
		\item If $\Xi$ and $\Upsilon$ are biorthogonal frames in $\mathcal{H}$.
	\end{enumerate}
	Under these conditions, the pair $\left(\{P \xi_i\}_{i=1}^{\infty}, \{Q \eta_i\}_{i=1}^{\infty}\right)$ establishes a Parseval biframe precisely when $Q P^* = I$.
	
	Furthermore, in the case where $\{e_k\}_{i=1}^{\infty}$ represents an orthonormal basis of $\mathcal{H}$, the sequence\\ $\left(\{U e_k\}_{i=1}^{\infty}, \{Qe_k\}_{i=1}^{\infty}\right)$ will form a Parseval biframe for $\mathcal{H}$ if and only if the equality $Q U^* = I$ holds true.
\end{corollary}
\begin{proof}
	Assume the first condition is met, indicating that $\{\xi_i\}_{i=1}^{\infty}$ and $\{\eta_i\}_{i=1}^{\infty}$ are dual frames in $\mathcal{H}$. This implies that for every $\xi \in \mathcal{H}$, 
	$$\xi = \sum_{i=1}^{\infty}\left\langle \xi, \xi_i\right\rangle_{\mathcal{A}}\eta_i. $$
	Consequently, 
	$$\langle \xi, \xi\rangle_{\mathcal{A}} = \sum_{i=1}^{\infty}\left\langle \xi, \xi_i\right\rangle_{\mathcal{A}}\langle \eta_i, \xi\rangle_{\mathcal{A}}. $$
	Thus, $(\Xi, \Upsilon)$ forms a Parseval biframe. According to Theorem \ref{T48} and given $Q P^* = I$, $(P\Xi, Q\Upsilon)$ is also a Parseval biframe.
	
	Now, if the second condition holds, meaning $\Xi = \{\xi_i\}_{i=1}^{\infty}$ and $\Upsilon = \{\eta_i\}_{i=1}^{\infty}$ are biorthogonal frames, then for $\xi \in \mathcal{H}$,
	$$
	\begin{aligned}
		\sum_{i=1}^{\infty}\left\langle \xi, \xi_i\right\rangle_{\mathcal{A}}\left\langle \eta_i, \xi\right\rangle_{\mathcal{A}} &= \sum_{i=1}^{\infty}\left\langle \xi, \xi_i\right\rangle_{\mathcal{A}}\left\langle \eta_i, \sum_{k=1}^{\infty}\left\langle \xi, S_{\Xi}^{-1} \xi_k\right\rangle \xi_k\right\rangle_{\mathcal{A}} \\
		&= \sum_{i=1}^{\infty} \sum_{k=1}^{\infty}\left\langle \xi, \xi_i\right\rangle_{\mathcal{A}}\left\langle S_{\Xi}^{-1} \xi_k, \xi\right\rangle_{\mathcal{A}}\left\langle \eta_i, \xi_k\right\rangle_{\mathcal{A}} \\
		&= \sum_{i=1}^{\infty} \sum_{k=1}^{\infty}\left\langle \xi, \xi_i\right\rangle_{\mathcal{A}}\left\langle S_{\Xi}^{-1} \xi_k, \xi\right\rangle_{\mathcal{A}} \delta_{i, k} \\
		&= \sum_{i=1}^{\infty}\left\langle \xi, \xi_i\right\rangle_{\mathcal{A}}\left\langle S_{\Xi}^{-1} \xi_i, \xi\right\rangle_{\mathcal{A}} \\
		&= \langle S_{\Xi} S_{\Xi}^{-1} \xi, \xi\rangle_{\mathcal{A}} \\
		&= \langle \xi, \xi\rangle_{\mathcal{A}}.
	\end{aligned}
	$$
	
	Thus, $(\Xi, \Upsilon)$ is a Parseval biframe, and similar to the first conclusion, $(P\Xi, Q\Upsilon)$ is also a Parseval biframe.
	
	Finally, by Theorem \ref{T48}, since $\{e_i\}_{i=1}^{\infty}$ is an orthonormal basis for $\mathcal{H}$, $(\{e_i\}_{i=1}^{\infty}, \{e_i\}_{i=1}^{\infty})$ is a Parseval biframe, implying that $(\{Pe_i\}_{i=1}^{\infty}, \{Qe_i\}_{i=1}^{\infty})$ is a Parseval biframe as well.
\end{proof}

\section*{Acknowledgments}
The authors are thankful to the area editor and referees for giving valuable comments and suggestions
Declarations

\vspace{0.75cm}
\noindent\textbf{\textsf{Declaration:}} This paper has been accepted for publication in the Montes Taurus Journal of Pure and Applied Mathematics (MTJPAM). 

\vspace{0.75cm}
\noindent\textbf{\textsf{Author Contributions:}} The authors equally conceived of the study, participated in its
design and coordination, drafted the manuscript, participated in the
sequence alignment, and read and approved the final manuscript. 

\vspace{0.2cm}
\noindent\textbf{\textsf{Conflicts of Interest:}} The authors declare that they have no competing interests.

\vspace{0.2cm}
\noindent\textbf{\textsf{Funding (Financial Disclosure):}} The authors declare that there is no funding available for this paper.


\vspace{0.2cm}
\begin{thebibliography}{9}
	
\bibitem{Aza} M. M. Azandaryani, A. Fereydooni, Pair frames in Hilbert $ C^{\ast}- $modules. Proceedings - Mathematical Sciences, 128(2)(2018). doi:10.1007/s12044-018-0396-1 

\bibitem{Chi} X. Chi, P. Li, Approximate Oblique Dual g-Frames for Closed Subspaces of Hilbert Spaces, Mediterr. J. Math. (2023) 20:220. 

\bibitem{DGM} I. Daubechies, A. Grossmann,  Y. Meyer, \emph{Painless non orthogonal expansions}, J. Math.
Phys. {\bf 27}  (1986), 1271-1283.			

\bibitem{Duf} R. J. Duffin, A. C. Schaeffer, \emph{A class of non harmonic fourier series}, Trans. Am. Math. Soc. {\bf 72} (1952),
341--366.

\bibitem{Amr}A. El Amrani, M. Rossafi, T. El krouk, K-Riesz bases and K-g-Riesz bases in Hilbert $ C^{\ast}- $module, Proyecciones (Antofagasta) 42(05):1241-1260.

\bibitem{Fer} A. Fereydooni, A. Safapour, Pair frames, Results Math., 66 (2014) 247–263

\bibitem{Lar1} M. Frank, D. R. Larson, \emph{$\mathcal{A}$-module frame concept for Hilbert $\mathcal{C}^{\ast}$-modules}, functinal and harmonic analysis of
wavelets, Contempt. Math. 247 (2000), 207-233.

\bibitem{KAI} G. Kaiser,  \emph{A Friendly Guide to Wavelets}, Birkh\"{a}user, Boston, 1994.

\bibitem{Kal}I. Kaplansky, Modules Over Operator Algebras,Am. J. Math.75(1953), 839-858.

\bibitem{Koc} M. R. Kouchi and A. Rahimi, On controlled frames in Hilbert $ C^{\ast}- $modules, Int. J. Walvelets Multi. Inf. Process., 15(4) (2017): 1750038.



\bibitem{Li} D. Li and J. Leng, Operator representations of g-frames in Hilbert spaces, Linear Multilin.Alg. 68(9)(2020), 1861-1877.

\bibitem{FR1} F. D. Nhari, R. Echarghaoui, M. Rossafi, $K-g-$fusion frames in Hilbert $C^{\ast}-$modules, Int. J. Anal. Appl. 19 (6) (2021).

\bibitem{Pas} W. Paschke, Inner product modules over B* -algebras. Trans. Am. Math. Soc. 182, 443–468 (1973).

\bibitem{Rah} M. Rahmani, G-Frames Generated by Iterated Operators, Sahand Communications in Mathematical Analysis (SCMA) Vol. 20 No. 4 (2023), 243-260.

\bibitem{RFDCA} M. Rossafi, FD. Nhari, C. Park, S. Kabbaj, Continuous g-Frames with $C^{\ast}$-Valued Bounds and Their Properties. Complex Anal. Oper. Theory 16, 44 (2022). https://doi.org/10.1007/s11785-022-01229-4

\bibitem{r1} M. Rossafi, S. Kabbaj, $\ast$-K-operator Frame for $\operatorname{End}_{\mathcal{A}}^{\ast}(\mathcal{H})$,  Asian-Eur. J. Math. 13 (2020), 2050060.

\bibitem{r3} M. Rossafi, S. Kabbaj, Operator Frame for $\operatorname{End}_{\mathcal{A}}^{\ast}(\mathcal{H})$, J. Linear Topol. Algebra, 8 (2019), 85-95.

\bibitem{r5} M. Rossafi, S. Kabbaj, $\ast$-K-g-frames in Hilbert $\mathcal{A}$-modules, J. Linear Topol. Algebra, 7 (2018), 63-71.

\bibitem{r6} M. Rossafi, S. Kabbaj, $\ast$-g-frames in tensor products of Hilbert $C^{\ast}$-modules, Ann. Univ. Paedagog. Crac. Stud. Math. 17 (2018), 17-25.

\bibitem{Sah} N. K. Sahu, Controlled g-frames in Hilbert $ C^{\ast}- $modules, Mathematical Analysis and its Contemporary Applications, 3 (2021), 65–82.

\bibitem{S}	W. Sun, G-frames and g-Riesz bases. J. Math. Anal. App. 322,437-452 (2006).	

\bibitem{Xia}X. Xiao, X. Zeng, Some properties of g-frames in Hilbert $ C^{\ast}- $modules, J. Math. Anal. Appl. 363 (2010) 399–408.

\end{thebibliography}
\end{document}